\documentclass[12pt]{amsart}

\usepackage{amsthm,amsfonts}
\usepackage{amssymb,graphicx,color}

\newtheorem{theorem}{Theorem}[section]

\newtheorem{remark}[theorem]{Remark}
\newtheorem{conjecture}[theorem]{Conjecture}

\begin{document}

\title[]{Ambient Lipschitz equivalence of real surface singularities}

\author[L.~Birbrair]{Lev Birbrair*}\thanks{*Research supported under CNPq 302655/2014-0 grant and by Capes-Cofecub}
\address{Dept Matem\'atica, Universidade Federal do Cear\'a
(UFC), Campus do Picici, Bloco 914, Cep. 60455-760. Fortaleza-Ce,
Brasil} \email{birb@ufc.br}

\author[A.~Gabrielov]{Andrei Gabrielov**}\thanks{**Research supported by the NSF grants DMS-1161629 and DMS-1665115 }
\address{Dept Mathematics, Purdue University, 150 N. University Street, West Lafayette, IN 47907-2067, USA } \email{gabriea@purdue.edu}

\begin{abstract}
We present a series of examples of pairs of singular semialgebraic surfaces (germs of real semialgebraic sets of dimension two) in ${\mathbb R}^3$ and ${\mathbb R}^4$ which are bi-Lipschitz equivalent with respect to the outer metric, ambient topologically equivalent, but not ambient Lipschitz equivalent. For each singular semialgebraic surface $S\subset {\mathbb R}^4$, we construct infinitely many semialgebraic surfaces which are bi-lipschitz equivalent with respect to the outer metric, ambient topologically equivalent to $S$, but pairwise ambient Lipschitz non-equivalent.
\end{abstract}

\maketitle

\section{Introduction}

There are three different classification questions in Lipschitz Geometry of Singularities. The first question is the classification of singular sets with respect to the inner metric, where the distance between two points of a set $X$ is counted as an infimum of the lengths of arcs inside $X$ connecting the two points. The equivalence relation is the bi-Lipschitz equivalence with respect to this metric. The second equivalence relation is the bi-Lipschitz  equivalence defined by the outer metric, where the distance is defined as the distance in the ambient space. It is well known that the two classifications are not equivalent. For example, all germs of irreducible complex curves are inner bi-Lipschitz equivalent, but the question of the outer classification is much more complicated (see Pham-Teissier \cite{PT} and Fernandes \cite{F}).
Here we consider another natural equivalence relation. Two germs of semialgebraic sets are called ambient Lipschitz equivalent if there exists a germ of a bi-Lipschitz homeomorphism of the ambient space transforming the germ of the first set to the germ of the second one.
Two outer bi-Lipschitz equivalent sets are always inner bi-Lipschitz equivalent, but can be ambient topologically non-equivalent (see Neumann-Pichon \cite{NP}).
The main question of the paper is the following. Suppose we have two germs of semialgebraic sets bi-Lipschitz equivalent with respect to the outer metric. Suppose that the germs are ambient topologically equivalent. Does it imply that the sets are ambient Lipschitz equivalent?
In this paper, we present four examples of the germs of surfaces for which the answer is negative. The surfaces in Examples 1, 2 and 3 are ambient topologically equivalent and bi-Lipschitz equivalent with respect to the outer metric, but their tangent cones at the origin are not ambient topologically equivalent. By the theorem, recently proved by Sampaio \cite{S}, ambient Lipschitz equivalence of two sets implies ambient Lipschitz equivalence of their tangent cones. Thus the sets in our three examples cannot be ambient Lipschitz equivalent.
In Example 4, the tangent cones of the two surfaces at the origin are ambient topologically equivalent.
The argument in that case is more delicate and requires a special ``broken bridge'' construction.
The last part of the paper is devoted to the proof of the main theorem of the paper: For any germ of a semialgebraic surface $S$ in ${\mathbb R}^4$ there exist infinitely many semialgebraic surfaces, such that all these surfaces are ambient topologically equivalent to $S$, bi-Lipschitz equivalent with respect to the outer metric, but any two of them are not ambient Lipschitz equivalent. To prove this theorem, we generalize the broken bridge construction of Example 4.

The question on the relation of these classifications was posed to us by Alexandre Fernandes and Zbigniew Jelonek. We thank them for posing the question.
We would like to thank Anne Pichon for her comments and suggestions.

\section{Examples in $\mathbb{R}^3$}

\noindent{\bf Example 1.} Consider semialgebraic sets $X_1$ and $X_2$ in $\mathbb{R}^3 $ (see Fig.~\ref{fig:example1}) defined by the following equations and inequalities: \vspace{0,3cm}
\begin{equation}\label{X1}
\begin{array}{ccc}
  X_1 = \big\{ \left( \left( x^2-2xt+y^2 \right) \left( x^2+2xt+y^2 \right) - t^k \right)\times\\

\left( { \left( x - t \right) }^2 + { \left( y - \frac{t}{2} \right) }^2 - \frac{t^2}{16} \right) \left( { \left( x - t \right) }^2 + { \left( y + \frac{t}{2} \right) }^2 - \frac{t^2}{16} \right) = 0,\\

  t \geq 0\big\}.
 \end{array}
\end{equation}
\begin{equation}\label{X2}
\begin{array}{ccc}
X_2 = \big\{ \left( \left( x^2 - 2xt + y^2 \right) \left( x^2 + 2xt + y^2 \right) - t^k \right)\times\\

\left( { \left( x - t \right) }^2 + { y  }^2 - \frac{t^2}{16} \right) \left( { \left( x + t \right) }^2 + {  y  }^2 - \frac{t^2}{16} \right) = 0, \\

 t \geq 0\big\}.
 \end{array}
\end{equation}
\noindent Here $k>4$ is an integer.

\begin{figure}
 \centering
 \includegraphics[width=3.5in]{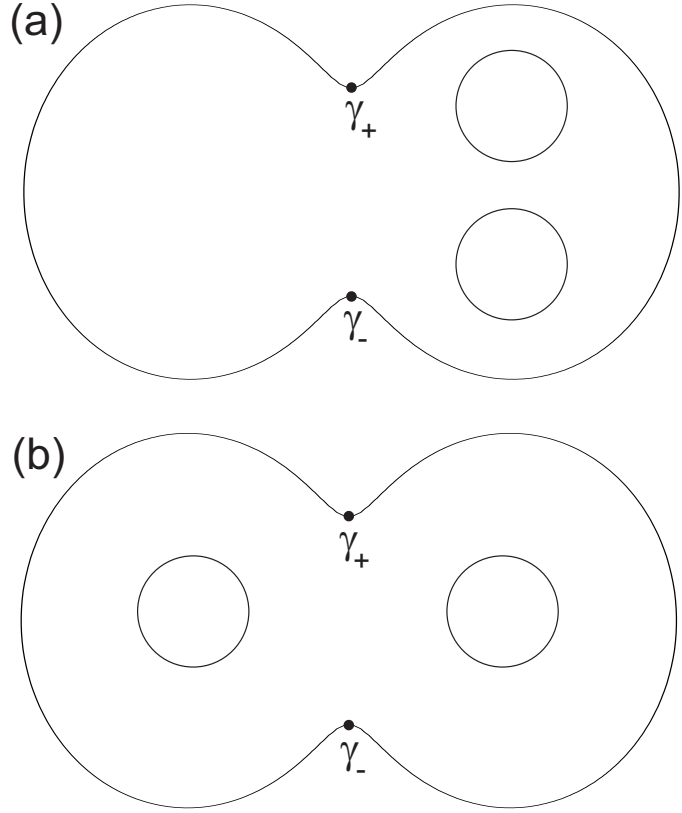}
 \caption{Links of the surfaces (a) $X_1$ and (b) $X_2$ in Example 1. }\label{fig:example1}
 \end{figure}

\begin{theorem}\label{th:example1}
 The germs at the origin of the surfaces $X_1$ and $X_2$ are bi-Lipschitz equivalent with respect to the outer metric, ambient topologically equivalent, but not ambient Lipschitz equivalent.
\end{theorem}

\begin{proof}
Notice that $X_1 = U_1\, \cup\, U_2\, \cup\, U_3$, where
$$U_1=\left\{\left( {\left( x - t \right) }^2 + y^2 - t^2 \right) \left( { \left( x + t \right) }^2 + y^2 - t^2 \right) = t^{k},\; t \geq 0\right\}$$
and the sets \ $U_2 ,\, U_3$ \ are straight cones over the circles \ $ {( x - 1)}^2 + {( y - \frac{1}{2})}^2 = \frac{1}{16}$ \ and \  $ ( x - 1)^2 + ( y + \frac{1}{2})^2 = \frac{1}{16}.$
The set $X_2$ is the union of $U_1$  and the sets $V_2 ,\, V_3$ which are straight cones over the circles \ $( x - 1 )^2 + y^2 = \frac{1}{16}$ \ and \ ${( x + 1 )}^2 + y^2 = \frac{1}{16}.$

Notice that \ $U_2 ,\, U_3 ,\, V_2$ \ and \ $V_3$ \ are linearly (thus bi-Lipschitz) equivalent. In particular, there exist invertible linear maps $\varphi:U_2\to V_2$ and $\psi: U_3\to V_3$
(one can define $\varphi(x,y,t)=(x,y-\frac{t}2,t)$ and $\psi(x,y,t)=(x-2t,y+\frac{t}2,t)\;$).
Observe that $U_2\cup U_3$ and $V_2\cup V_3$ are normally embedded. Moreover, there exist positive constants \ $c_1 , c_2$ \ such that for any point \ $ p = ( x, y , t)\in U_2\cup U_3\cup V_2\cup V_3$ one has
$c_1 t \ < \ d ( p, U_1 ) \ < \ c_2 t$. Thus the map $\phi:X_1\to X_2$ defined as

\vspace{0,5cm}
$\phi(p) = $ $\left\{
\begin{array}{cccc}
  \ \ \ p \ \ \ \ \mbox{if}\ \ \ p \in U_1 \\
  \varphi(p) \ \ \mbox{if} \ \ \ p \in U_2 \\
  \psi(p) \ \ \mbox{if} \ \ \ p \in U_3 \\
  \end{array}
 \right.
$
\vspace{0,5cm}

\noindent is bi-Lipschitz with respect to the outer metric.

The sets $X_1$ and $X_2$ are ambient topologically equivalent, each of them being equivalent to a cone over the union of three disjoint circles in the plane $t=1$,
two of them bounding non-intersecting discs inside a disc bounded by the third one.

However, the tangent cones to $X_1$ and $X_2$, defined by the homogeneous parts of degree $4$ of (\ref{X1}) and (\ref{X2}), are not ambient topologically equivalent.
The tangent cone of $X_1$ at the origin is the union of $U_2$, $U_3$ and a straight cone $W$ over two tangent circles $( x - 1 )^2 + y^2 = 1$ and $( x + 1 )^2 + y^2 =1$ in the plane $t=1$,
with the cones $U_2$ and $U_3$ inside one of the two circular cones of $W$, while the tangent cone of $X_2$ is the union of $V_2$, $V_3$ and $W$, with the cones $V_2$ and $V_3$ inside two
different cones of $W$. Thus $X_1$ and $X_2$ are not ambient bi-Lipschitz equivalent, by the theorem of Sampaio \cite{S}.
This happens, of course, because $U_1$ is not normally embedded, with the arcs $\gamma_+$ and $\gamma_-$ in $U_1\ \cap \{x=0\}$ (see Fig.~\ref{fig:example1}) having the tangency order $k/4 > 1$.
\end{proof}

\begin{figure}
 \centering
 \includegraphics[width=4in]{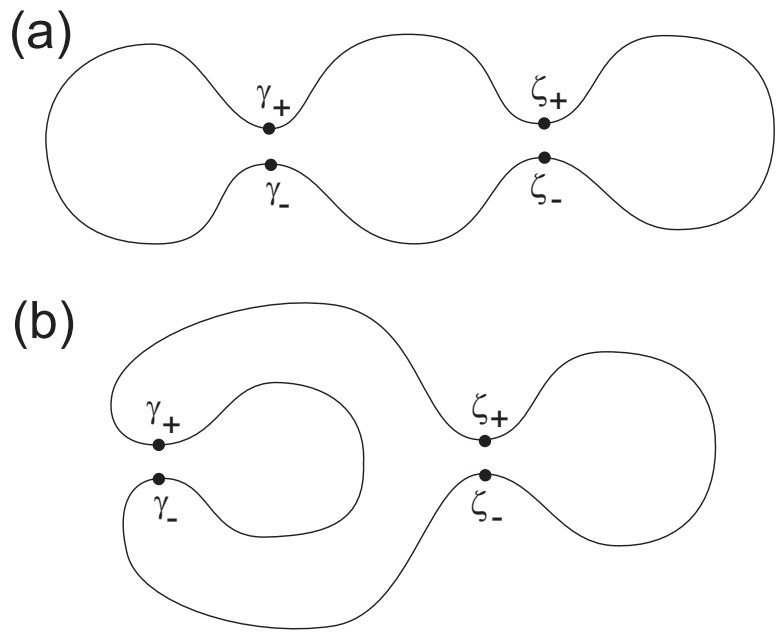}
 \caption{Links of the surfaces (a) $X_1$ and (b) $X_2$ in Example 2.}\label{fig:example2}
 \bigskip
 \end{figure}

\begin{figure}
 \centering
 \includegraphics[width=4in]{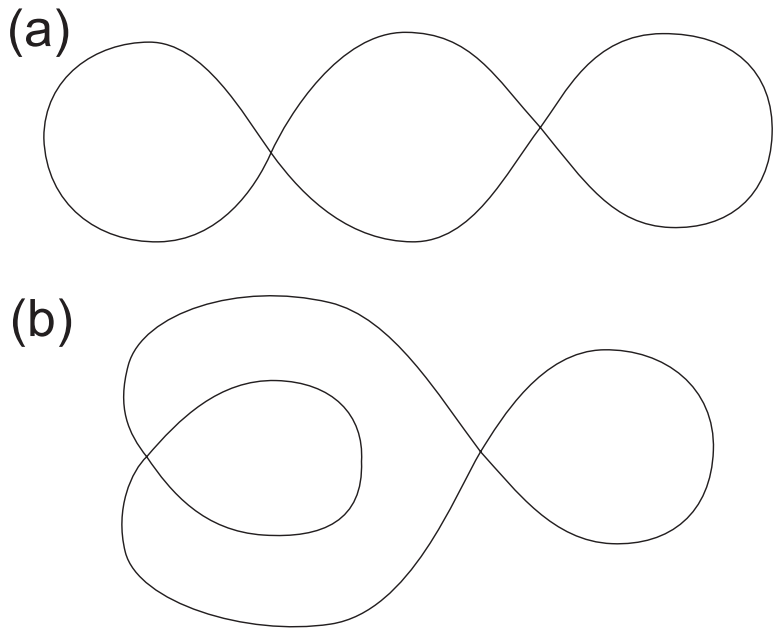}
 \caption{Links of the tangent cones at the origin of the surfaces (a) $X_1$ and (b) $X_2$ in Example 2.}\label{fig:example2-tangentcone}
 \end{figure}

{\bf Example 2.} Let $X_1$ and $X_2$ be semialgebraic surfaces in $\mathbb R^3$ with the links at the origin shown in Fig.~\ref{fig:example2}, and tangent cones at the origin
as in Fig.~\ref{fig:example2-tangentcone}. One can define $X_1$ and $X_2$ by explicit semialgebraic formulas, similarly to the method employed in Example 1.
Both surfaces $X_1$ and $X_2$ are ambient topologically equivalent to a cone over a circle. These surfaces are bi-Lipschitz equivalent with respect
to the outer metric, but not ambient Lipschitz equivalent by Sampaio's theorem, since their tangent cones are not ambient topologically equivalent.
The arguments are similar to those in Example 1.

\section{Examples in ${\mathbb R}^4$}

{\bf Example 3.} Let $H \subset \mathbb{R}^4$ be a surface defined as follows:
\begin{equation*}
\left\{
 y^2 - x^2 = { ( x^2 + y^2 - 2 t^2 )}^{2},\; z=0,\; |y| \leq t \leq 1
 \right\}.
\end{equation*}
The surface $H$ has two branches, tangent at the origin. It is bounded by the straight lines
$$l_1 = ( z = 0 ,\; y = x = t ), \quad l_2 = ( z = 0,\; y = -x = t ),$$
$$l_3 = ( z = 0,\; y = -x = -t ),\quad l_4 = ( z = 0,\; y = x = -t).$$
The tangent cone of $H$ at the origin is the surface
$$\left\{ y = \pm x,\; z = 0,\; |y| \le t \right\}.$$
The link of $H$ (more precisely, the section of $H$ by the plane $\{z = 0, \; t = 1/8\}\;$) is shown in Fig.~\ref{fig:H}.
The arcs ${\gamma}_{+} $ and ${\gamma}_{-}$ are tangent at the origin.

Let $K_1 , K_2 , K_3$ be nontrivial knots in $\mathbb R^3$ such that $K_3$ is a connected sum of $K_1$ and $K_2$.
Let $X_1$ be a surface in $\mathbb{R}^4$, obtained as follows.

Consider a smooth semialgebraic embedding $\widetilde{K}_1$ of the knot $K_1$ to the hyperplane $\{t=1\}$ in $\mathbb{R}^4_{x,y,z,t}$.
Suppose that $\widetilde{K}_1$ contains the points $(1, 1, 0, 1)\in l_1$ and $(1, -1, 0, 1)\in l_3$, and that $\widetilde{K}_1 \cap H$ contains only these points. Let $s_1 \subset \widetilde{K}_1$ be the segment connecting the points $(1, 1, 0, 1)$ and $(1, -1, 0, 1)$ such that replacing this segment by a straight line segment does not change the embedded topology of $\widetilde{K}_1$. Let $\widetilde{K}_2$ be a smooth semialgebraic realization of $K_2$, in the same hyperplane of $\mathbb{R}^4$.
Suppose that $\widetilde{K}_2$ contains the points $( - 1, 1, 0,1)\in l_2$ and $(-1, -1, 0, 1)\in l_4$, and that a segment $s_2$ of $\widetilde{K}_2$ connecting these points may be replaced by a straight line segment without changing the embedded topology of $\widetilde{K}_2$. Suppose that $\widetilde{K}_2 \cap  H $ contains only the points $(-1 , 1 , 0, 1)$ and $( - 1, -1, 0, 1)$, and that $ \widetilde{K}_2 \cap \widetilde{K}_1 =  \emptyset$.

Let $K_1'$ be the straight cone over $\widetilde{K}_1 - {s}_1$ and let $K_2'$ be the straight cone over $\widetilde{K}_{2}-{s}_2$. Let $X_1 = K_1' \cup H \cup K_2'.$ The link of the set $X_1$ is shown in Fig.~\ref{fig:knot1}a.

\begin{figure}
 \centering
 \includegraphics[width=3in]{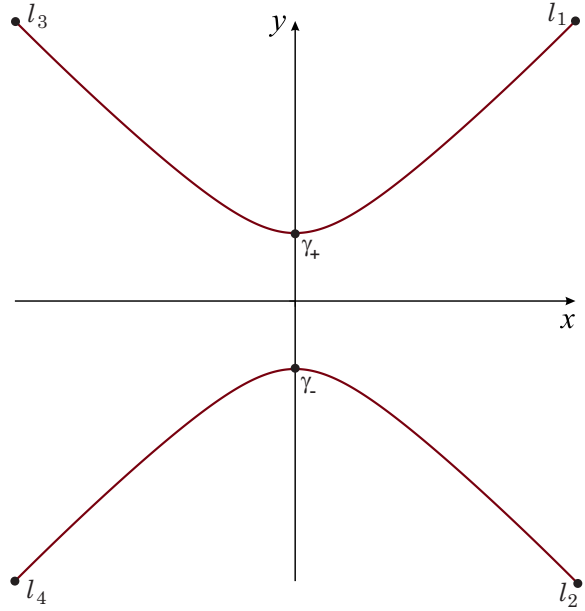}
 \caption{Link of the surface $H$ in Example 3.}\label{fig:H}
 \end{figure}

\begin{figure}
 \centering
 \includegraphics[width=4.5in]{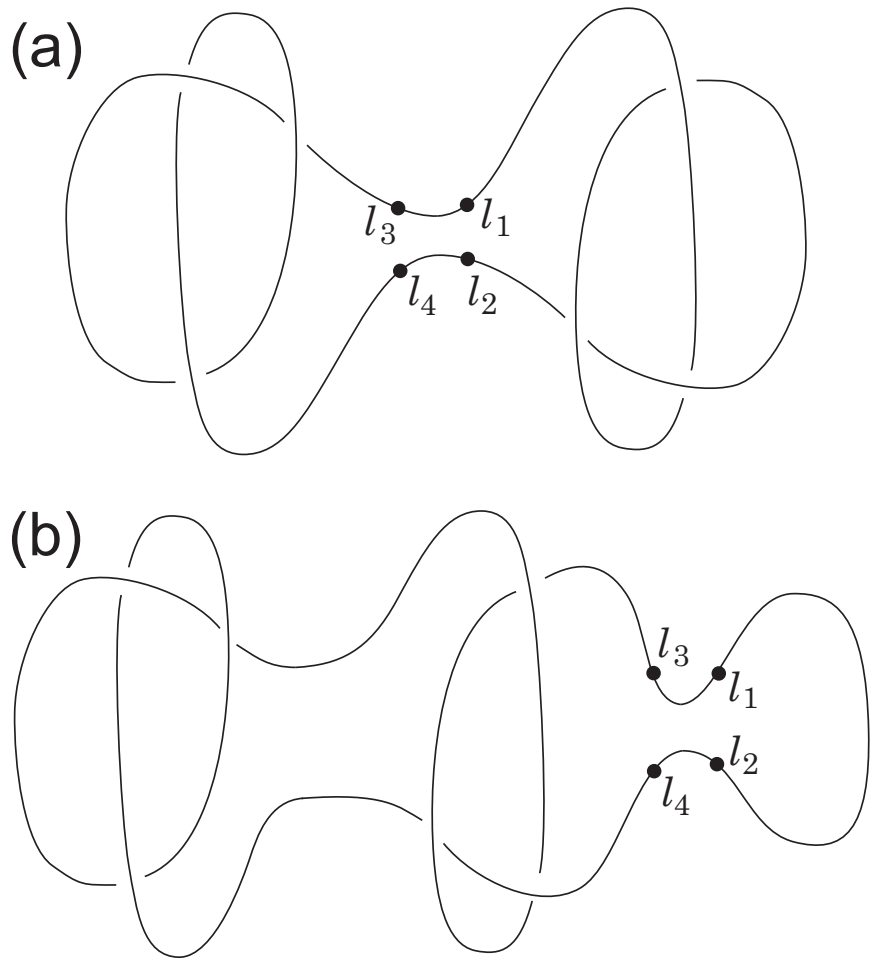}
 \caption{Links of the surfaces (a) $X_1$ and (b) $X_2$ in Example 3.}\label{fig:knot1}
 \end{figure}

Let us define the set $X_2$ using the same construction as above, with the knot $K_1$ replaced by $K_3$ and the knot $K_2$ by the unknotted circle $K_4$.
We assume, as before, that a smooth semialgebraic realisation $\widetilde{K}_3$ of $K_3$ contains points $(1, 1, 0, 1)$ and $(1, -1, 0, 1)$, that $\widetilde{K}_3 \cap H$ contains only these points, and that replacing the segment $s_3$ of $\widetilde{K}_3$ connecting these points by a straight line segment does not change the embedded topology of $\widetilde{K}_3$. Similar assumptions are made about a smooth semialgebraic embedding
 $\widetilde{K}_4$ of $K_4$ and its segment $s_4$ connecting the points $(-1 , 1 , 0, 1)$ and $( - 1, -1, 0, 1)$.
Let $K_3'$ and $K_4'$ be the straight cones over $\widetilde{K}_3 - s_3$ and $\widetilde{K}_4 - s_4$. Let $X_2 = K_3' \cup H \cup K_4'$ (see Fig.~\ref{fig:knot1}b).

\begin{theorem}\label{thm:ex3} The germs of the sets $X_1$ and $X_2$ at the origin are bi-Lipschitz equivalent with respect to the outer metric, ambient topologically equivalent, but not ambient bi-Lipschitz equivalent.
\end{theorem}

\begin{figure}
 \centering
 \includegraphics[width=4.5in]{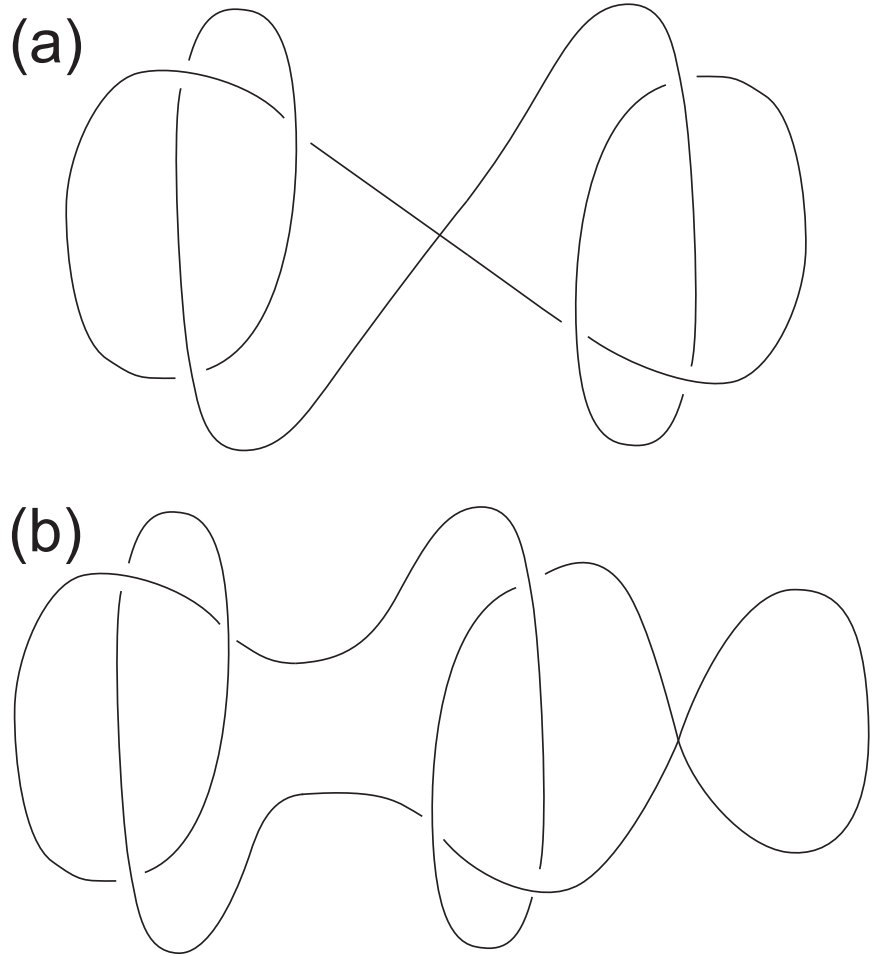}
 \caption{Links of the tangent cones at the origin of the surfaces (a) $X_1$ and (b) $X_2$ in Example 3.}\label{fig:knot1t}
 \end{figure}

\begin{proof} Since $ \widetilde{K}_1 ,\, \widetilde{K}_2 ,\, \widetilde{K}_3 ,\, \widetilde{K}_4$ are smooth, the corresponding cones $K_1'$, $K_2'$, $K_3'$, $K_4'$ are normally embedded and bi-Lipschitz equivalent with respect to the outer metric. The bi-Lipschitz maps $${\phi}_1 : K_1' \rightarrow K_3'\ \ \mbox{and}\ \ {\phi}_2 : K_2' \rightarrow K_4'$$ can be chosen in such a way that
$$ {\phi}_1 ( K_1' \cap \{ t = c \} ) \ = \ K_3' \cap \{ t = c \} \ \ \mbox{ for all } \ c > 0, \ \ \mbox{and} $$
$$ {\phi}_2 ( K_2' \cap \{ t = c \} ) \ = \ K_4' \cap \{ t = c \} \ \ \mbox{ for all } \ c > 0. $$

Then one can define the map $\phi$ as follows:

\vspace{0,5cm}
$\phi(x) = $ $\left\{
\begin{array}{cccc}
  {\phi}_1 (x) \ \ \ \ \ \mbox{if} \ \ \ x \in K_1' \\
 \ \ \ x \ \ \ \ \ \   \mbox{if} \ \ \ x \in H \\
  {\phi}_2 (x)  \ \ \ \ \ \mbox{if} \ \ \ x \in K_2' \\
  \end{array}
 \right.
$ \vspace{0,5cm}

Clearly, the map $\phi$ is a bi-Lipschitz map. The surfaces $X_1$ and $X_2$ are ambient topologically equivalent because their links are knots equivalent to $K_3$.
From the other hand, the corresponding tangent cones at the origin are not ambient topologically equivalent: the tangent cone of $X_1$ is a straight cone over the union of $K_1$ and $K_2$, pinched at one point
(see Fig.~\ref{fig:knot1t}a), while
the tangent cone of $X_2$ is a straight cone over the union of $K_3$ and the unknotted circle, pinched at one point (see Fig.~\ref{fig:knot1t}b).

By the theorem of Sampaio \cite{S}, the surfaces $X_1$ and $X_2$ are not ambient bi-Lipschitz equivalent.
\end{proof}

{\bf Example 4.}

For $1\le \beta<q$, define the set $A_{q,\beta}=T_+\cup T_-\subset {\mathbb R}^4$, where
$$T_\pm=\left\{ 0\le t\le 1,\; -t^\beta\le x\le t^\beta,\; y=\pm t^q,\; z=0\right\}$$
are two normally embedded $\beta$-H\"older triangles tangent at the origin with the tangency exponent $q$.
The set $A_{q,\beta}$ is called a $(q,\beta)$-bridge (see Fig.~\ref{fig:bridge}, left). The boundary of $A_{p,q}$ consists of the four arcs $$\{t\ge 0,\; x=\pm t^\beta,\; y=\pm t^q\}.$$
For some $p$ such that $\beta<p<q$, let $B_{p,q,\beta}$ be the set obtained from $A_{q,\beta}$ by removing from $T_+$ the $p$-H\"older triangle bounded by the arcs $\{t\ge 0,\; x=\pm t^p,\; y=t^q,\; z=0\}$,
and from $T_-$ the $p$-H\"older triangle bounded by the arcs $\{t\ge 0,\; x=\pm t^p,\; y=-t^q,\; z=0\}$), and replacing them by two $q$-H\"older triangles
$$\left\{ 0\le t\le 1,\; x=t^p,\; -t^q\le y \le t^q,\; z=0\right\}\ \ \mbox{and}$$
$$\left\{ 0\le t\le 1,\; x=-t^p,\; -t^q\le y \le t^q,\; z=0\right\}.$$
The set $B_{p,q,\beta}$ is called a broken $(q,\beta)$-bridge (see Fig.~\ref{fig:bridge}, right).

\begin{figure}
 \centering
 \includegraphics[width=4.5in]{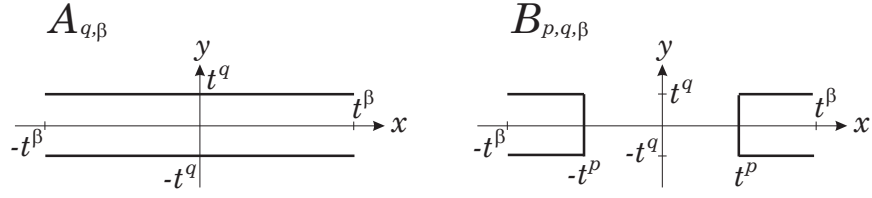}
 \caption{Links of a $(q,\beta)$-bridge $A_{q,\beta}$ and a broken $(q,\beta)$-bridge $B_{p,q,\beta}$.}\label{fig:bridge}
 \end{figure}

\begin{figure}
 \centering
 \includegraphics[width=4.5in]{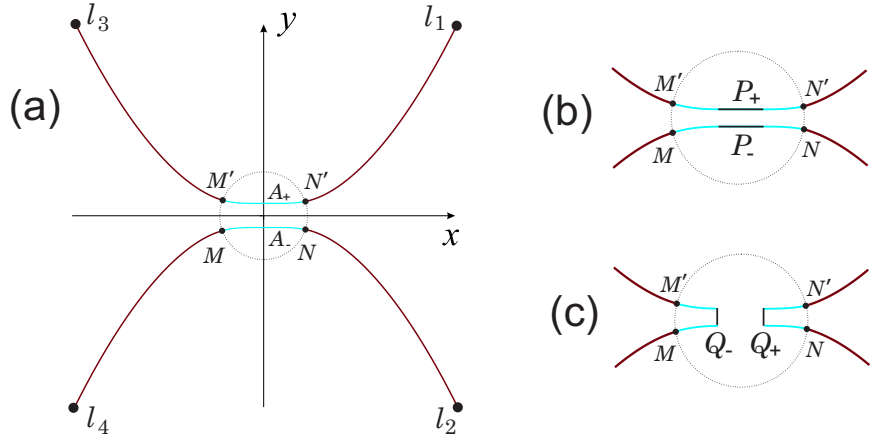}
 \caption{(a) Link of the surface $G$ in Example 4. (b) H\"older triangles $P_+$ and $P_-$. (c) Broken $(3,2)$-bridge $B$ with the H\"older triangles $Q_+$ and $Q_-$.}\label{fig:G}
 \end{figure}

Let $G \subset \mathbb{R}^4$ be a surface defined as follows:
\begin{equation*}
\left\{
 y^2 t^2 - x^4 = { ( x^2 + y^2 - 2 t^2 )}^{4},\; z=0,\; \mid y \mid\; \leq t \leq 1
 \right\}.
\end{equation*} \vspace{0,1cm}
The surface $G$ has two branches, tangent at the origin. It is bounded by the straight lines
$$l_1 = ( z = 0 ,\; y = x = t ), \quad l_2 = ( z = 0,\; -y = x = t ),$$
$$l_3 = ( z = 0,\; y = -x = t ),\quad l_4 = ( z = 0,\; y = x = -t).$$
The tangent cone of $G$ at the origin is the surface
$$\left\{ y^2 t^2 = x^4,\; z = 0,\; |y| \leq t \right\}.$$
The link of $G$ (more precisely, the section of $G$ by the plane $\{z = 0, \; t = 1/8\}\;$) is shown in Fig.~\ref{fig:G}a.
The intersection of $G$ with any surface $\left\{ x=c t^\mu \right\}$, where $\mu\ge 2$,
consists of two arcs having the tangency order $3$.
Thus $G$ contains a subset $A$, consisting of two normally embedded $2$-H\"older triangles $A_+$ and $A_-$ (see Fig.~\ref{fig:G}a where $[M,N]$ is the link of $A_-$ and $[M',N']$ is the link of $A_+$), which is ambient bi-Lipschitz equivalent to a $(3,2)$-bridge.
It is easy to check that such a subset $A$ is unique up to a bi-Lipschitz homeomorphism of $\mathbb R^4$ preserving $G$.

Consider two trivial knots $K_0$ and $K_1$ embedded in the hyperplane $\{t=1\}\subset \mathbb{R}^4_{x,y,z,t}$ as shown in Fig.~\ref{fig:knot2}a and Fig.~\ref{fig:knot2}b.
Suppose that each of these two knots contains the four points
$$(1, 1, 0, 1)\in l_1,\;(1, -1, 0, 1)\in l_2,\;( - 1, 1, 0,1)\in l_3,\;(-1, -1, 0, 1)\in l_4,$$
and that the intersection of each of the two knots
with the ball $U$ of radius $\sqrt{2}$ in $\{t=1\}$ consists of two unlinked segments $s_1$ and $s_2$ connecting $(1, 1, 0, 1)$ with $(-1, 1, 0, 1)$ and $( 1, -1, 0, 1)$ with $(-1, -1, 0, 1)$, respectively,
as shown in Figs.~\ref{fig:knot2}a and \ref{fig:knot2}c, where $U$ is shown as a dotted circle.
We assume also that the union of the two segments $s_1$ and $s_2$ coincides with $G\cap\{t=1\}$.

We define the surface $X_0$ as the union of $G$ and a straight cone over $K_0\setminus U$, and the surface $X_1$ as the union of $G$ and a straight cone over $K_1\setminus U$.

\begin{figure}
 \centering
 \includegraphics[width=4.5in]{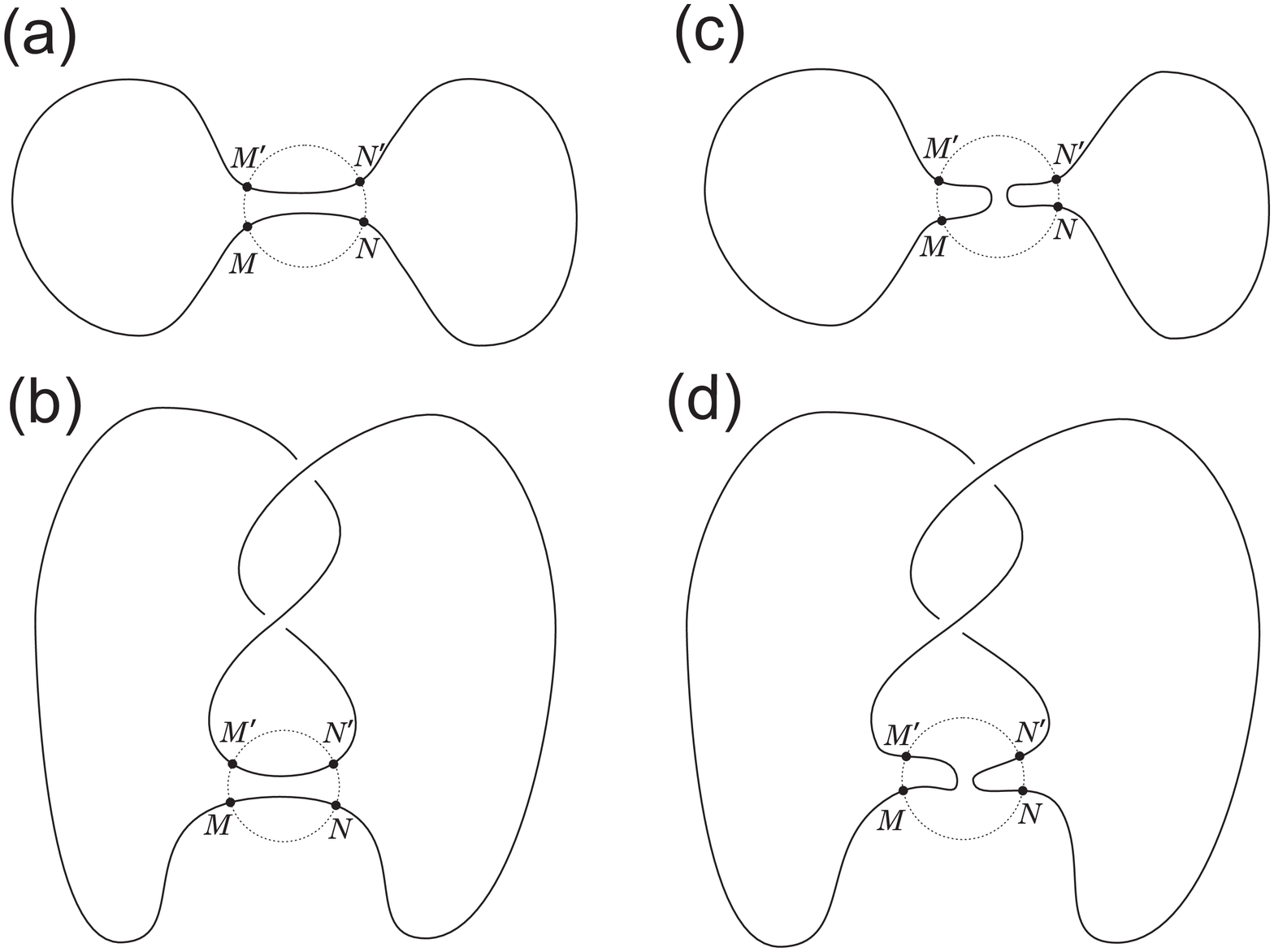}
 \caption{Links of the surfaces (a) $X_0$ and (b) $X_1$ in Example 4.}\label{fig:knot2}
 \end{figure}

\begin{theorem}\label{thm:ex4}
The germs of the surfaces $X_0$ and $X_1$ at the origin are bi-Lipschitz equivalent with respect to the outer metric, ambient topologically equivalent, but not ambient Lipschitz equivalent.
\end{theorem}

\begin{proof}
Suppose that $X_0$ and $X_1$ are ambient Lipschitz equivalent.
Let $h:\mathbb R^4\to\mathbb R^4$ be a bi-Lipschitz homeomorphism such that $h(X_0)=X_1$.
The set $A'=h(A)$ is ambient Lipschitz equivalent to a $(3,2)$-bridge.
For any arc $\gamma\subset A'$ there is an arc $\gamma'\subset A'$
such that the inner distance in $X_1$ between $\gamma$ and $\gamma'$ has exponent 1, but the outer distance between them has exponent 3.
No such arcs exist outside $G$. Due to the uniqueness of a $(3,2)$-bridge in $G$ up to ambient Lipschitz equivalence, there is
a bi-Lipschitz homeomorphism $h'$ of $\mathbb R^4$ preserving $G$ and mapping $A'$ to $A$.
Moreover, we may assume $h'$ to be identity outside $U$, thus $h'(X_1)=X_1$.
Combining $h$ with $h'$, we may assume that $h(A)=A$.

For $p\in(2,3)$, let $P\subset A$ be the union of two $p$-H\"older triangles $P_+$ and $P_-$ that should be removed from $A$ and replaced by two $q$-H\"older triangles $Q_+$ and $Q_-$
to obtain a broken $(3,2)$-bridge $B$ (see Figs.~\ref{fig:G}b and \ref{fig:G}c).
Define the surface $\widetilde{X}_0$ (see Fig.~\ref{fig:knot2}c) by replacing $A\subset X_0$ with $B$, and
the surface $\widetilde{X}_1=h(\widetilde{X}_0)$ (see Fig.~\ref{fig:knot2}d) by replacing $A=h(A)\subset X_1$ with $h(B)$.
Then $\widetilde{X}_0$ and $widetilde{X}_1$ are not ambient topologically equivalent: the link of $\widetilde{X}_0$ consists of two unlinked circles, while the link of $\widetilde{X}_1$ consists of two linked circles.
This contradicts our assumption that $X_0$ and $X_1$ are ambient Lipschitz equivalent.
\end{proof}

\begin{remark}
Notice that the tangent cones of $X_0$ and $X_1$ are ambient topologically equivalent to a cone over two unknotted circles, pinched at one point.
Thus Sampaio's theorem does not apply, and we need the ``broken bridge'' construction in this example.
Notice also that the broken bridge construction employed in this example allows one to construct examples
(both in $\mathbb{R}^3$ and in $\mathbb{R}^4$) of outer bi-Lipschitz equivalent, ambient topologically equivalent
but ambient Lipschitz non-equivalent surface germs with the tangent cones as small as a single ray.
\end{remark}

\begin{conjecture} Let $(S_0,0)$ and $(S_1,0)$ be two normally embedded real semialgebraic surface germs which are ambient topologically equivalent and
bi-Lipschitz equivalent with respect to either inner or outer metric (the two metrics are equivalent for normally embedded sets).
Then $S_0$ and $S_1$ are ambient Lipschitz equivalent.
\end{conjecture}

\section{Main result}

\begin{theorem}\label{thm:main} For any semialgebraic surface germ $(S,0) \subset \mathbb{R}^4$ there exist infinitely many semialgebraic surface germs $(X_i,0) \subset \mathbb{R}^4$ such that  \vspace{0,2cm}

$1) \ $ For all  $i$, the germs $( X_i , 0)$ are ambient topologically equivalent to $(S,0)$; \vspace{0,1cm}

$2) \ $ All germs $( X_i , 0)$ are bi-Lipschitz equivalent with respect to the outer metric; \vspace{0,1cm}

$3) \ $ The tangent cones of all germs $X_i$ at the origin are ambient topologically equivalent; \vspace{0,1cm}

$4) \ $ For $ i \neq j$ the germs $X_i$ and $X_j$ are not ambient bi-Lipschitz equivalent.
\end{theorem}

\begin{figure}
 \centering
 \includegraphics[width=3in]{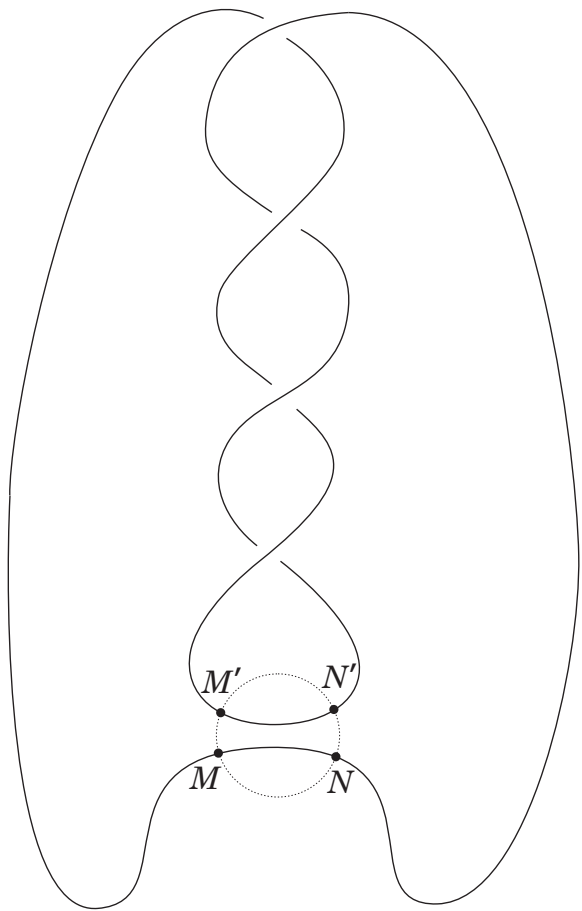}
 \caption{Link of the surface $X_2$ in Step 1.}\label{fig:knot2b}
 \end{figure}

 \begin{figure}
 \centering
 \includegraphics[width=4.5in]{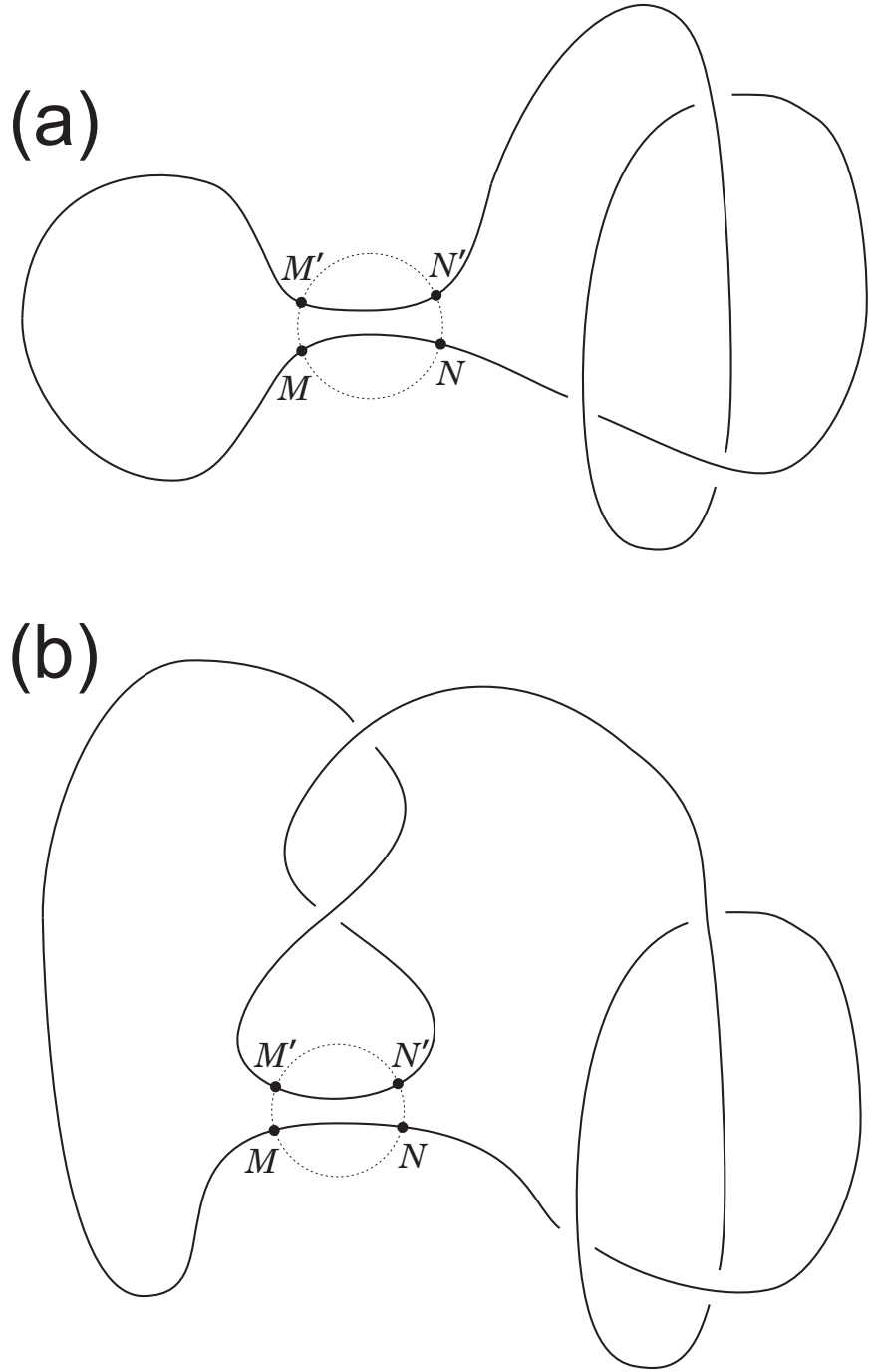}
 \caption{Links of the surfaces $Y_0$ and $Y_1$ in Step 2.}\label{fig:knot3}
 \end{figure}

{\bf Step 1} \ Consider first the case where the link $L$ of $S$ is an unknotted circle.
Our construction would be a modification of Example 4.
The germ $(X_i,0)$ is obtained from the surface $G$ considered in Example 4 (see Fig.~\ref{fig:G}a)
by attaching to it the straight cone over two segments in such a way that the braid connecting the pair of
points $(M,N)$ with the pair of points $(M',N')$ has $i$ twists.
The germs $(X_0,0)$ and $(X_1,0)$ are exactly those considered in Example 4. Their links are shown  in Fig.~\ref{fig:knot2}a and Fig.~\ref{fig:knot2}b.
The link of the germ $(X_2,0)$ is shown in Fig.~\ref{fig:knot2b}.
All these links are unknotted, thus all the surfaces $X_i$ are ambient topologically equivalent to the straight cone over the unknotted circle.

The same arguments as in the proof of Theorem \ref{thm:ex4} show that all germs $(X_i,0)$ are bi-Lipschitz equivalent with respect to the outer metric.
We now construct a bi-Lipschitz map $f_{ij}:(X_i,0)\to(X_j,0)$ which is the identity on the surface $G$. Since the complements of $G$ in $X_i$ and $X_j$ are straight cones, the map $f_{ij}$ on the cones can be defined as the conical extension of a bi-Lipschitz map on the links. The tangent cone of $X_i$ at the origin is ambient topologically equivalent to the cone over two unknotted circles pinched at a point.

From the other hand, the broken bridge construction described in Example 4 transforms $X_i$ into a set $\widetilde{X}_i$ with the link consists of two circles having the linking number $i$, same as the  number of twists of the braid connecting $(M, N)$ with $(M',N')$. This implies that $X_i$ is not ambient bi-Lipschitz equivalent to $X_j$ for $ i \neq j$. \vspace{0,2cm}

{\bf Step 2 } \ Consider the case where the link $L$ of the surface $S$ has a subset $L'$ which is a non-trivial knot.
Consider the link $L_i$ of the surface $X_i$ constructed in Step 1.
Since it is unknotted, its connected sum with $L'$ is ambient topologically equivalent to $L'$.
We define the surface $Y_i$ so that its link is the link $L$ of $S$ with $L'$ replaced by the connected
 sum of $L'$ and $L_i$. The surface $X'$ (except the cone over a segment of its link where $L'$
 is attached) is a subset of $Y_i$, and its complement in $Y_i$ is a cone over the link $L$ of $S$ (except the cone
 over a segment of $L'$ where it is attached to $L_i$). The links $L_0$ and $L_1$ from Example 4 with the trefoil knot attached are shown in Fig.~\ref{fig:knot3}a and Fig.~\ref{fig:knot3}b.

 By the same arguments as in Step 1 we obtain the conclusion of Theorem \ref{thm:main} in this case.\vspace{0,2cm}

{\bf Step 3} \ Consider the case where the link $L$ of $S$ is homeomorphic to a segment.
Consider a germ $( X_i , 0)$, defined in Step 1. Let $T_{\beta} \subset X_i$ be a H\"older triangle with the vertex at the origin, a subset of the conical part of $X_i$.

Let $ ( Z_i , 0 ) = ( X_i , 0) - T_{\beta}$. We claim that the germs $( Z_i, 0)$ satisfy the conditions \ $1,\, 2,\, 3,\, 4$ \ of Theorem \ref{thm:main}. The conditions $1$ and $2$ are evidently satisfied. The condition $3$ is true because the tangent cones of $X_i$and $Z_i$ at the origin are the same.

Suppose that $h_{ij} : ( \mathbb{R}^4 ,0) \rightarrow ( \mathbb{R}^4 ,0)$ is a bi-Lipschitz map,
such that $h_{ij}(Z_i, 0) = (Z_j , 0)$,  where $ i\ne j$. Let us apply the map $h_{ij}$ to $( X_i , 0)$. Notice that the set $ h_{ij}( X_i )$ is different from $X_i$ only in a small $\beta$-horn containing the H\"older triangle $T_{\beta}$.

If we apply the broken bridge construction to $h_{ij}( X_i )$, we get the same linking number for the two components as in Theorem 3.2.

Since the link of any surface germ either has a connected component homeomorphic to a segment,
or contains a subset which is a (possibly trivial) knot, this completes the proof of Theorem \ref{thm:main}.


\begin{thebibliography}{99}

\bibitem{F} A.~Fernandes. Topological equivalence of complex curves and bi-Lipschitz maps. The Michigan Mathematical Journal, Michigan, v. 51, n.3, p. 593-606, 2003.
\bibitem{NP} W.~Neumann and A.~Pichon, Lipschitz geometry does not determine embedded topological type, Journal of Singularities, v 10 , 2014 , p. 225-234.
\bibitem{PT} P.~Pham and B.~Teissier, Saturation Lipschitzienne d'une algèbre analytique complexe et saturation de Zariski. Prépublication Ecole Polytechnique 1969.
\bibitem{S} J.~E.~Sampaio. Bi-Lipschitz homeomorphic subanalytic sets have bi-Lipschitz homeomorphic tangent cones. Selecta Mathematica-New Series, v. 22, p. 553-559,

\end{thebibliography}
\end{document}